\documentclass[a4paper, 11pt]{amsart}
\usepackage{charter}
\usepackage{amsmath, amssymb, amsfonts, amsthm, hyperref, bm, graphicx, overpic}

\newcommand{\supp}[1]{\operatorname{supp}(#1)}
\newcommand{\epsi}{\varepsilon}

\newcommand{\norm}[1]{\lVert#1\rVert}
\newcommand{\numbersystem}[1]{\mathbb{#1}}

\newcommand{\N}{\numbersystem{N}}
\newcommand{\abs}[1]{\lvert#1\rvert}

\theoremstyle{plain}
\newtheorem*{theorem}{Theorem}
\newtheorem{lemma}{Lemma}

\newtheorem{proposition}[lemma]{Proposition}

\theoremstyle{definition}
\newtheorem*{definition}{Definition}

\begin{document}

\bibliographystyle{amsplain}

\title{Simultaneous packing and covering in sequence spaces}
\author{Konrad J. Swanepoel} 
\address{Fakult\"at f\"ur Mathematik,
	Technische Universit\"at Chemnitz,  
	D-09107 Chemnitz, Germany}
\email{konrad.swanepoel@gmail.com}
\dedicatory{Dedicated to the memory of Victor Klee}

\begin{abstract}
We adapt a construction of Klee (1981) to find a packing of unit balls in $\ell_p$ ($1\leq p<\infty$) which is efficient in the sense that enlarging the radius of each ball to any $R>2^{1-1/p}$ covers the whole space.
We show that the value $2^{1-1/p}$ is optimal.
\end{abstract}

\maketitle

\section{Introduction}
The so-called simultaneous packing and covering constant of a convex body $C$ in Euclidean space is a certain measure of the efficiency of a packing or a covering by translates of $C$.
This notion was introduced in various equivalent forms by Rogers \cite{Rogers1950}, Ry\v{s}kov \cite{Ryvskov1974} and L.~Fejes~T\'oth \cite{Fejes1976}, and in the lattice case can be traced back to Delone \cite{Delone1937}.
Its study has recently been given renewed attention by Zong \cite{Zong2002a, Zong2002, Zong2002b, Zong2003, Zong2008} and others \cite{Henk2005, Schurmann2006}.
Important contributions to the non-lattice case have also been made by Linhardt \cite{Linhart1978}, B\"or\"oczky \cite{Boroczky1986}, and Doyle, Lagarias and Randall \cite{Doyle1992}.
Since this notion avoids the use of density, it can be used to study packings and coverings in hyperbolic spaces or infinite dimensional spaces.
In this paper we determine the exact value of this constant for the $\ell_p$ spaces where $1\leq p<\infty$.
The main ingredient in the proof is an adaptation of a construction of Klee \cite{Klee1981}.

\section{The simultaneous packing and covering constant}
Let $(X,\norm{\cdot})$ be any normed space.
Denote the closed ball with center $x\in X$ and radius $r$ by $B(x,r)$.
A subset $P\subseteq X$ is (the set of centers of) \emph{an $r$-packing} if the collection of balls $\{B(x,r): x\in P\}$ are pairwise disjoint.
Equivalently, $P$ is \emph{$2r$-dispersed}, i.e., $d(x,y)> 2r$ for all distinct $x, y\in P$.
For any $P\subseteq X$, define \[r(P):=\sup\{r: \text{$P$ is an $r$-packing}\}.\]

A subset $P\subseteq X$ is (the set of centers of) \emph{an $R$-covering} (or \emph{$R$-net}) if the collection of balls $\{B(x,R): x\in P\}$ cover $X$, i.e., $X=\bigcup_{x\in P} B(x,R)$.
For any $P\subseteq X$, define \[R(P):=\inf\{R:\text{$P$ is an $R$-covering}\}.\]
Then $R(P)$ is the supremum of the radii of balls disjoint from $P$:
\begin{equation}\label{one}
R(P)=\sup\{R: \text{for some $x\in X$, $B(x,R)\cap P=\emptyset$}\}.
\end{equation}
If $P$ is an $r$-packing, then $R(P)-r$ is the supremum of the radii of balls that are disjoint from $\bigcup_{p\in P} B(p,r)$, and if $P$ is an $R$-covering, then $R-r(P)$ is the supremum of the radii of balls that are contained in more than one of $B(p,R)$, $p\in P$ \cite{Fejes1976}.

\begin{definition}
 The \emph{simultaneous packing and covering constant} of (the unit ball of) $X$ is
\[\gamma(X):=\inf\{R(P): \text{$P$ is a $1$-packing}\}.\]
\end{definition}
We could also have used $1$-coverings to define this constant, as shown by the identity \[\gamma(X)^{-1}=\sup\{r(P): \text{$P$ is a $1$-covering}\}.\]

It is clear that $R(P)\geq 1$ for any $1$-packing $P$.
By Zorn's lemma there always exists a maximal $1$-packing, which is necessarily a $2$-covering.
Therefore, \[1\leq\gamma(X)\leq 2.\]
If $X$ is finite-dimensional, then $\gamma(X)$ is exactly the simultaneous packing and covering constant of the unit ball of $X$, as discussed in the introduction.

\section{The main theorem}
The main result of the paper concerns the case $X=\ell_p$, $1<p<\infty$, which we recall is the space of real sequences $x=(x_i)_{i\in\N}$ such that $\sum_{i=1}^\infty\abs{x_i}^p<\infty$ with norm
\[ \norm{x}_p:=\left(\sum_{i=1}^\infty\abs{x_i}^p\right)^{1/p}.\]
\begin{theorem}
For each $p\in(1,\infty)$, $\gamma(\ell_p)=2^{1-1/p}$.
\end{theorem}
In particular, if $p$ is close to $1$, then $\gamma(\ell_p)$ is close to $1$, which means there are very good packings of unit balls in $\ell_p$.
Perhaps more surprisingly, if $p$ is very large, $\gamma(\ell_p)$ is close to $2$, i.e., any packing by unit balls has large holes.

In the next section we use a result of Burlak, Rankin and Robertson \cite{Burlak1958} to show the lower bound $\gamma(\ell_p)\geq 2^{1-1/p}$.
It is more difficult to find good packings.
In Section~\ref{klee} we adapt a construction of Klee \cite{Klee1981} to give packings that demonstrate the upper bound $\gamma(\ell_p)\leq 2^{1-1/p}$.
In fact, Klee already obtained this bound  for $\ell_p(\kappa)$ where $\kappa$ is a regular cardinal such that $\kappa^{\aleph_0}=\kappa$.
In our case, $\kappa=\aleph_0$, and then his construction has to be modified substantially.

\section{The lower bound}
For a proof of the following packing property of $\ell_p$, see \cite{Burlak1958}, \cite{Kottman1970}, or \cite{Wells1975}.
\begin{lemma}\label{BRR}
If the unit ball of $\ell_p$ contains an infinite $\alpha$-dispersed set, then $\alpha\leq 2^{1/p}$.
\end{lemma}

To prove  $\gamma(\ell_p)\geq 2^{1-1/p}$, it is sufficient to show the following.
\begin{proposition}
Let $P$ be a $1$-dispersed subset of $\ell_p$ where $1\leq p<\infty$.
Then $R(P)\geq 2^{-1/p}$.
\end{proposition}
\begin{proof}
Let $0<\epsi<R(P)$.
Set $r:=R(P)-\epsi$ and $\delta:=((r+2\epsi)^p-r^p)^{1/p}$.
By \eqref{one} there exists $c\in\ell_p$ such that $B(c,r)\cap P=\emptyset$.
Translate $P$ by $-c$ so that we may assume without loss of generality that $c=o$.
Thus $\norm{x}>r$ for all $x\in P$.

We claim that $Q:=B(o,r+\delta+2\epsi)\cap P$ is infinite.
Suppose to the contrary that $Q$ is finite.
As usual, we denote by $e_n$ the sequence which is $1$ in position $n$ and $0$ in all other positions.
For any $n\in\N$ and $x\in Q$,
\[ \norm{x-\delta e_n}_p^p = \norm{x}_p^p-\abs{x_n}^p+\abs{x_n-\delta}^p.\]
Therefore,
\[ \lim_{n\to\infty}\norm{x-\delta e_n}_p^p=\norm{x}_p^p+\delta^p > r^p+\delta^p=(r+2\epsi)^p.\]
Since $Q$ is finite, there exists $n\in\N$ such that for all $x\in Q$,
$\norm{x-\delta e_n}>r+2\epsi$.
On the other hand, for any $x\in P\setminus Q$, $\norm{x}> r+\delta+2\epsi$, and then the triangle inequality gives $\norm{x-\delta e_n}> r+2\epsi$.
Therefore, $B(\delta e_n,r+2\epsi)\cap P=\emptyset$, which gives
$R(P)\geq r+2\epsi$, a contradiction.

Thus $Q$ is infinite, and by Lemma~\ref{BRR},
\[R(P)+\delta+\epsi=r+\delta+2\epsi\geq 2^{-1/p}.\]
By letting $\epsi\to 0$, we obtain $R(P)\geq 2^{-1/p}$, as required.
\end{proof}

\section{Constructing an optimal packing}\label{klee}
To prove the upper bound $\gamma(\ell_p)\leq 2^{1-1/p}$, it is sufficient to show the following.
\begin{proposition}
For any $1\leq p<\infty$ there exists a $2^{1/p}$-dispersed set $P\subseteq\ell_p$ such that $R(P)=1$.
\end{proposition}
\begin{proof}
We recursively construct the set $P$ together with the space, which in the end is isometric to $\ell_p$.

If $A$ is any set, we denote by $\ell_p(A)$ the normed space of all real-valued functions $f$ on $A$ with countable support $\supp{f}:=\{a\in A: f(a)\neq 0\}$ such that
$\sum_{a\in \supp{f}}\abs{f(a)}^{p}<\infty$,
and with norm \[\norm{f}_p:=\Bigl(\sum_{a\in \supp{f}}\abs{f(a)}^{p}\Bigr)^{1/p}.\]
Thus $\ell_p=\ell_p(\N)$ is isometric to $\ell_p(A)$ if $A$ is countably infinite.
For any $a\in A$, let $e_a$ be the function on $A$ such that $e_a(a)=1$ and $e_a(b)=0$ for all $b\in A$, $b\neq a$.
If $A\subseteq A'$, then we consider $\ell_p(A)$ to be a subspace of $\ell_p(A')$ in the natural way.

We construct two sequences of countable sets $P_n$ and $D_n$.
Let $P_1=\emptyset$ and $D_1=\{0\}=\ell_p(\emptyset)$.
If $P_1,\dots,P_n$ and $D_1,\dots,D_n$ have been constructed for some $n\geq 1$, let \[P_{n+1}:=\{x+e_x:x\in D_n\}\subseteq\ell_p\Bigl(\bigcup_{i=1}^n D_i\Bigr),\]
and let $D_{n+1}$ be a countable dense subset of
\[\ell_p\Bigl(\bigcup_{i=1}^{n} D_i\Bigr)\setminus\bigcup\Bigl\{B(x,1): x\in\bigcup_{i=1}^{n+1}P_i\Bigr\}.\]
By the definition of $P_{n+1}$ it follows that $D_k\subseteq\bigcup_{x\in P_{k+1}}B(x,1)$ for each $k=1,\dots,n$, hence $D_{n+1}$ is disjoint from $\bigcup_{i=1}^n D_i$.
It follows that the $P_n$ are also pairwise disjoint.

Let $P:=\bigcup_{n\in\N}P_n$.
Then $P$ is a subset of the space $\ell_p(\bigcup_{n\in\N}D_n)$, which is isometric to $\ell_p$ (note that already $D_2$ is infinite).
We now show that $P$ is $2^{1/p}$-dispersed and is a $(1+\epsi)$-covering for all $\epsi>0$. 

Choose two arbitrary elements $x+e_x, y+e_y\in P$, where $x\in D_n$ and $y\in D_m$, $x\neq y$, and $n\leq m$.
Since $\supp{x}, \supp{y}\subseteq\bigcup_{i=1}^{m-1} D_i$, which is disjoint from $D_m$, it follows that $\supp{x-y}$ and $\supp{e_y}=\{y\}$ are disjoint.
We distinguish between two cases.

If $n=m$, then $\supp{e_x}=\{x\}$ is also disjoint from $\supp{x-y}$ and $\supp{e_y}$, hence
\[ \norm{(x+e_x)-(y+e_y)}_p^p=\norm{x-y+e_x-e_y}_p^p=\norm{x-y}_p^p+1+1>2.\]

In the second case, $n<m$.
Since $y\in D_m$, and $x+e_x\in P_{n+1}$, it follows that $y\notin B(x+e_x,1)$, hence $\norm{x+e_x-y}_p^p>1$.
Since $\supp{x+e_x-y}$ and $\supp{e_y}$ are now disjoint,
\[ \norm{(x+e_x)-(y+e_y)}_p^p=1+\norm{x+e_x-y}_p^p>1+1.\]
It follows that $P$ is $2^{1/p}$-dispersed.

Let $\epsi>0$ and choose an arbitrary $x\in\ell_p(\bigcup_{n\in\N} D_n)$.
Choose $N\in\N$ large enough such that $\norm{x-y}_p<\epsi/2$ for some $y\in\ell_p(\bigcup_{i=1}^{N-1} D_i)$.
If $y\in\bigcup\{B(z,1):z\in\bigcup_{i=1}^{N}P_i\}$, then for some $z\in\bigcup_{i=1}^{N}P_i$, \[\norm{x-z}_p\leq\norm{x-y}_p+\norm{y-z}_p<1+\epsi/2.\]
If on the other hand $y\notin\bigcup\{B(z,1):z\in\bigcup_{i=1}^{N}P_i\}$, then there exists $z\in D_N$ such that $\norm{y-z}_p^p<(1+\epsi/2)^p-1$.
Note that $z+e_z\in P_{N+1}$.
Since $z\in D_N$ and $y,z\in\ell_p(\bigcup_{i=1}^{N}D_i)$, $\supp{e_z}=\{z\}$ is disjoint from $\supp{y-z}$.
Thus
\[ \norm{z+e_z-y}_p^p=1+\norm{z-y}_p^p<(1+\epsi/2)^p,\]
and by the triangle inequality, $\norm{z+e_z-x}<1+\epsi$.
It follows that $P$ is a $(1+\epsi)$-covering. 
\end{proof}


\end{document}